\documentclass[12pt,a4paper]{amsart}
\usepackage[utf8]{inputenc}

\usepackage{amssymb,amsmath,amsthm}
\usepackage[abbrev,alphabetic]{amsrefs}
\usepackage{tikz-cd}
\usetikzlibrary{arrows}
\usepackage{graphicx}
\usepackage[hidelinks]{hyperref}
\usepackage{mathtools}
\usepackage[titletoc,title]{appendix}
\usepackage{graphicx}
\usepackage{caption}
\usepackage{subcaption}
\usepackage{blindtext}
\usepackage[titletoc,title]{appendix}
\usepackage[english]{babel}
\usepackage{mathrsfs}
\usepackage{microtype}
\usepackage{xcolor}
\usepackage{stmaryrd}
\usepackage[shortlabels]{enumitem}

\usepackage[capitalize]{cleveref}

\newtheorem{theorem}{Theorem}[section]
\newtheorem{lemma}[theorem]{Lemma}
\crefname{lemma}{Lemma}{Lemmata}
\newtheorem{corollary}[theorem]{Corollary}
\newtheorem{proposition}[theorem]{Proposition}
\theoremstyle{definition}
\newtheorem{definition}[theorem]{Definition}
\newtheorem{remark}[theorem]{Remark}
\newtheorem{example}[theorem]{Example}
\newtheorem{conjecture}[theorem]{Conjecture}
\newtheorem*{theorem*}{Theorem}
\newtheorem*{conjecture*}{Conjecture}

\newtheoremstyle{maintheorem}{}{}{\itshape}{}{\bfseries}{}{.5em}{#1 \thmnote{#3}.}
\theoremstyle{maintheorem}
\newtheorem*{mainthm}{Theorem}

\newcommand{\isom}{\cong}

\newcommand{\Z}{\mathbb{Z}}
\newcommand{\Q}{\mathbb{Q}}

\newcommand{\normal}[1]{\langle\!\langle #1 \rangle\!\rangle}

\newcommand{\D}{\mathcal D}
\newcommand{\F}{\mathbb F}
\def\immerses{\looparrowright}
\def\injects{\hookrightarrow}

\DeclareSymbolFontAlphabet{\amsmathbb}{AMSb}

\DeclareMathOperator{\bs}{BS}
\DeclareMathOperator{\cd}{cd}

\DeclareMathOperator{\fp}{FP}

\DeclareMathOperator{\core}{Core}
\DeclareMathOperator{\comm}{Comm}

\DeclareMathOperator{\cat}{CAT}
\DeclareMathOperator{\Hom}{Hom}

\newcounter{dawidcomments}

\newcounter{marcocomments}

\title{Virtually free-by-cyclic groups}
\author{Dawid Kielak}
\author{Marco Linton}

\address{University of Oxford, Oxford, OX2 6GG, UK}

\email{kielak@maths.ox.ac.uk}

\address{University of Oxford, Oxford, OX2 6GG, UK}

\email{marco.linton@maths.ox.ac.uk}

\begin{document}

\begin{abstract}
We obtain a homological characterisation of virtually free-by-cyclic groups among groups that are hyperbolic and virtually compact special. As a consequence, we show that many groups known to be coherent actually possess the stronger property of being virtually free-by-cyclic. In particular, we show that all one-relator groups with torsion are virtually free-by-cyclic, solving a conjecture of Baumslag.
\end{abstract}

\maketitle

\section{Introduction}

A group $G$ is \emph{coherent} if its finitely generated subgroups are finitely presented. There are currently no known examples of coherent hyperbolic groups of cohomological dimension four or greater; on the contrary, many such groups are known to be incoherent. For example, Kapovich has shown that every arithmetic hyperbolic manifold of simple type and dimension at least four has incoherent fundamental group \cite{mkapovich_13}. In his recent survey on coherence, Wise suggests that coherence could essentially be a low-dimensional phenomenon among hyperbolic groups \cite{wise_21}. Nevertheless, since small cancellation groups can be incoherent by the Rips construction \cite{rips_82}, even understanding which hyperbolic groups of dimension two are coherent remains a difficult task. 

A notable class of two-dimensional groups known to be coherent is formed by  \emph{free-by-cyclic groups}, whose coherence is due to Feighn--Handel \cite{feighn_99} -- throughout this article, as in that of Feighn--Handel, we do not require the free kernel to be finitely generated.
Our main result characterises virtually free-by-cyclic groups among all hyperbolic and virtually compact special groups in terms of rational cohomological dimension $\cd_{\Q}$ and $L^2$-Betti numbers $b^{(2)}_n$. Using this characterisation, we show that several classes of two-dimensional hyperbolic groups already known to be coherent actually turn out to be virtually free-by-cyclic.

\begin{theorem}
\label{main_2}
Let $H$ be a hyperbolic and virtually compact special group. The following are equivalent:
\begin{enumerate}[(a)]
    \item\label{itm:zero_betti} $\cd_{\Q}(H)\leqslant 2$ and $b_2^{(2)}(H) = 0$,
\item\label{itm:subgroup of fbc} $H$ has a finite index subgroup that is a subgroup of a \{finitely generated free\}-by-cyclic group,
\item \label{itm:vfbc} $H$ is virtually free-by-cyclic.
\end{enumerate}
\end{theorem}

Since subgroups of free-by-cyclic groups are free-by-cyclic, \ref{itm:subgroup of fbc} implies \ref{itm:vfbc} without any extra hypotheses. We prove in \cref{ascending_corollary} that \ref{itm:vfbc} implies \ref{itm:zero_betti} without the hyperbolic and virtually compact special assumption. However, \ref{itm:zero_betti} cannot imply \ref{itm:vfbc} in general: the Baumslag--Solitar groups $\bs(1, 2^k)$, with $k\geqslant 2$, are not free-by-cyclic, have rational cohomological dimension $2$, have vanishing second $L^2$-Betti number, and have all of their finite index subgroups isomorphic to $\bs(1, 2^n)$ for various values of $n\geqslant 2$.

It is less clear what happens if only hyperbolicity is dropped, but we still insist on the group being virtually compact special. Similarly, the case of $H$ being virtually RFRS is interesting and open.

\cref{main_2} is relatively easy to apply, since there is rich literature on computing $L^2$-homology. In particular, $L^2$-Betti numbers are known for one-relator groups thanks to Dicks--Linnell~\cite{dicks_07}, for three-manifold groups by the work of Lott--L\"uck~\cite{LottLueck1995}, and for torsion-free lattices in {connected linear} semi-simple Lie groups as determined by {Borel~\cite{Borel1985} (see also the work of Dodziuk~\cite{Dodziuk1979})}. There is also a wealth of results about vanishing of $L^2$-homology (especially in dimension $1$), see L\"uck's book~\cite{luck_02}*{Section 7}.
Finally, in the presence of the Atiyah conjecture, $L^2$-homology becomes a particular kind of group cohomology, and hence all the usual tools of group cohomology are at ones disposal. 

\subsection{Applications}

In 1967, Baumslag \cite{baumslag_67} asked whether all one-relator groups with torsion are residually finite. He later strengthened this by conjecturing that all one-relator groups with torsion should have the stronger property of being virtually free-by-cyclic \cite[Problem 6]{baumslag_86}. This conjecture was also stated by Fine and Rosenberger \cite[Conjecture 7.1]{fine_12} and restated by Baumslag several times \cite[Problem 4.3]{baumslag_99}, \cite[Conjecture 3.7]{baumslag_19}. Baumslag and Troeger proved that certain families of one-relator groups with torsion are virtually free-by-cyclic and even discuss a potential counterexample in \cite{baumslag_08}.

A variation of Baumslag's conjecture due to Wise asserts that all hyperbolic one-relator groups should be virtually free-by-cyclic \cite[Conjecture 17.8]{wise_21}. Since one-relator groups with torsion are hyperbolic by Newman's spelling theorem \cite{newman_68}, this encompasses Baumslag's conjecture. If hyperbolic one-relator groups are virtually compact special as suggested in \cite{wise_14}, \cref{one-relator_fbc} would resolve this more general conjecture.

Recently, Wise \cite{wise_21_quasiconvex} confirmed Baumslag's residual finiteness conjecture by  showing that one-relator groups with torsion are virtually compact special. {Combining this with the computation of the $L^2$-Betti numbers of one-relator groups due to Dicks--Linnell \cite{dicks_07}, we here confirm Baumslag's virtually free-by-cyclic conjecture.}

\begin{corollary}
One-relator groups with torsion are virtually free-by-cyclic.
\end{corollary}

{More generally, Dicks--Linnell showed in \cite{dicks_07} that if $S$ is the fundamental group of a compact surface and $w\in S$, then $\cd_{\Q}(S/\normal{w}) \leqslant 2$ and $b_2^{(2)}(S/\normal{w}) = 0$.}

\begin{corollary}
\label{one-relator_fbc}
Let $S$ be the fundamental group of a compact surface, possibly with boundary, and let $w\in S$. If $G = S/\normal{w}$ is hyperbolic and virtually compact special, then $G$ is virtually free-by-cyclic.
\end{corollary}

Many interesting insights were obtained via the \emph{negative immersions property}, first introduced by Wise \cite{wise_04}. Using a more general definition, Louder--Wilton \cite{louder_22} showed that in the context of one-relator groups, negative immersions is equivalent to the property of being \emph{$2$-free}. {Additionally, they showed that this is equivalent to the defining relation having primitivity rank at least three; that is, the defining relation does not lie as an imprimitive word in a two generator subgroup of the underlying free group}. Among one-relator groups generated by at least three elements, Louder--Wilton also show that this property is generic by work of Puder \cite{puder_15}. The second author showed in \cite{linton_22} that one-relator groups with negative immersions are hyperbolic and virtually compact special. We may now combine these observations with our main theorem, and obtain the following.

\begin{corollary}
\label{one-relator_neg}
One-relator groups with negative immersions are virtually free-by-cyclic.
\end{corollary}

One-relator groups with torsion and one-relator groups with negative immersions are already known to be coherent by \cite{louder_20,wise_20} and \cite{louder_21_uniform} respectively. We emphasize here that {\cref{one-relator_neg}} does not provide an alternative proof of coherence of one-relator groups with negative immersion as the results in \cite{louder_21_uniform} were used to establish hyperbolicity and virtual compact specialness in \cite{linton_22}. Sapir and \v{S}pakulov\'{a} prove in \cite{sapir_11} that generic one-relator groups, generated by at least three elements, are coherent by showing that they generically embed in ascending HNN-extensions of finitely generated free groups. Nevertheless, \cref{one-relator_fbc} slightly enlarges the class of one-relator groups known to be coherent. For instance, small cancellation one-relator groups are not necessarily in any of the above classes, but are virtually free-by-cyclic by \cref{one-relator_fbc}.

\begin{corollary}
\label{small_cancellation}
Small cancellation one-relator groups are virtually free-by-cyclic, and thus, coherent.
\end{corollary}

The two-generator case of \cref{small_cancellation} was established by Kropholler, Wilkes, and the first author \cite{kielak_22}, where it is also shown that one-relator groups are generically virtually free-by-cyclic. If $w(x, y)$ is a non-primitive small-cancellation word that is not a proper power, then the one-relator group $\langle a, b, c, d \mid w([a, b]^2, [c, d]^2)\rangle$ is small-cancellation, torsion-free, does not have negative immersions and is not contained in Sapir--\v{S}pakulov\'{a}'s generic set.

\smallskip

A class of free-by-cyclic one-relator groups of well-known geometric significance is formed by fundamental groups of closed orientable surfaces. If $S$ is such a  surface of genus $g$, then $\pi_1(S)$ acts geometrically on the hyperbolic plane with fundamental domain a compact right-angled regular $2g$-gon. More generally, the \emph{Bourdon building} \cite{bourdon_97} $X_{p, q}$ is the building whose chambers are hyperbolic right-angled $p$-gons (here $p\geqslant 5$) with edges of thickness $q\geqslant 2$ and with vertex links isomorphic to the complete bipartite graph $K_{q,q}$. Since $X_{p, q}$ is a $\cat(-1)$ space, any uniform lattice of $X_{p, q}$ is a hyperbolic group. Moreover, by putting a vertex in the middle of each chamber and connecting each such vertex to the midpoint of each edge, we see that $X_{p, q}$ also has a $\cat(0)$ square complex structure. Thus, any uniform lattice of $X_{p, q}$ is also virtually compact special by Agol's theorem \cite[Theorem 1.1]{agol_13}. In particular, \cref{main_2} now characterises virtually free-by-cyclic uniform lattices of $X_{p, q}$. In \cite{wise_21}, Wise proves that uniform lattices of $X_{p, q}$ are coherent if and only if $q< p-1$. One ingredient for this characterisation was Dymara's result that the second $L^2$-Betti number vanishes if and only if $q<p-1$ \cite[Theorem 4.2]{dymara_04}. Combined with our main theorem, we obtain the following.

\begin{corollary}
\label{bourdon}
Let $p\geqslant 5$ and $q\geqslant 2$ be integers. A uniform lattice of the Bourdon building $X_{p, q}$ is virtually free-by-cyclic if and only if $q<p-1$.
\end{corollary}

Haglund showed that for $p\geqslant 5$, $q\geqslant 2$ (for the $p = 5$ case, one also needs \cite{haglund_08}), all uniform lattices of $X_{p, q}$ are commensurable \cite{haglund_06}. By combining \cref{bourdon} with the quasi-isometric rigidity results obtained by Bourdon--Pajot \cite{bourdon_00}, we obtain infinitely many quasi-isometry classes of free-by-cyclic groups. {Here, just as in the rest of the article, we emphasise that the free group can be infinitely generated.}

{Thanks to another result of Bourdon \cite[Theorem A]{Bo16}, we may further generalise one direction of \cref{bourdon}.}

\begin{corollary}
\label{conformal_dimension}
{Let $G$ be a hyperbolic and virtually compact special group with $\cd_{\Q}(G)\leqslant 2$. If the Gromov boundary $\partial G$ has conformal dimension less than 2, then $G$ is virtually free-by-cyclic.}
\end{corollary}

{This is indeed a generalisation of one direction of \cref{bourdon} by Bourdon's computation of the conformal dimension of $\partial X_{p, q}$ \cite[Theorem 1.1]{bourdon_97}. To the best of the authors' knowledge, the conformal dimension of the Gromov boundary of a hyperbolic virtually free-by-cyclic group is not known, leaving open the possibility of a converse to \cref{conformal_dimension}.}

Mutanguha in  \cite{mutanguha_21} raises the question whether hyperbolic ascending HNN-extensions of free groups are virtually free-by-cyclic. We answer this question affirmatively, subject to the extra assumption of virtual compact specialness. Hagen--Wise in \cite{hagen_15} show that hyperbolic \{finitely generated free\}-by-cyclic groups are cubulable and thus, by \cite[Theorem 1.1]{agol_13}, are virtually compact special. If hyperbolic ascending HNN-extensions of free groups are also cubulable, as suggested in \cite{hagen_15}, the following corollary would complete the picture.

\begin{corollary}
\label{ascending}
Let $G$ be a finitely generated ascending HNN-extension of a free group. If $G$ is hyperbolic and virtually compact special, then $G$ is virtually free-by-cyclic.
\end{corollary}

Finally, we also point out the following, which is a consequence of \cite[Corollary 1.2]{wise_20_l2}.

\begin{corollary}
Let $X$ be a finite $2$-complex with $\pi_1(X)$ hyperbolic and compact special. If $H_2(X) = 0$, then $\pi_1(X)$ is virtually free-by-cyclic.
\end{corollary}

\subsection{Overview}

We first introduce the necessary background on agrarian and $L^2$-homology. Next, we proceed to prove an agrarian version of the Freiheitssatz, based on the $L^2$-Freiheitssatz of Peterson--Thom \cite{peterson_11}.

\begin{theorem}
	\label{freiheit_intro}
Let {$R$ be a ring, $G$ a group and $\D$ a skew-field into which $RG$ embeds}. Let $\{g_1, \dots, g_n\}$ be a generating set for $G$. There is a subset of the generators $\{h_1, \dots, h_k\}$ generating freely a free subgroup $H$ such that the restriction  map
	\[
	H^1(G, \mathcal{D})\to H^1(H, \mathcal{D})
	\]
	is an isomorphism.
\end{theorem}

The reasons for doing it in the generality of agrarian homology are two-fold: first, the proofs lie naturally in this setting, and second, \cref{main_2} should admit an agrarian generalisation, in which $L^2$-Betti numbers are replaced with agrarian Betti numbers. This cannot be done at present, since the theory of agrarian homology has not yet been sufficiently developed. The interest here lies in the possibility of using Jaikin-Zapirain's skew-fields in place of the Linnell skew-field, effectively giving a version of \cref{main_2} using a finite-fields version of $L^2$-Betti numbers. This is desirable for example from the perspective of profinite invariants.

In \cref{section_4}, we relate the second $L^2$-Betti number of HNN-extensions and amalgamated free products with the second $L^2$-Betti numbers of the base groups. These results lead naturally to the definition of an $L^2$-independent hierarchy. We show that several classes of groups appearing in the literature admit $L^2$-independent hierarchies, and we compute their $L^2$-Betti numbers.

In \cref{section_5} we provide a self-contained proof of the following result that may be of independent interest.

\begin{theorem}
\label{conjugacy_separated_subgroup_intro}
Let $n$ be a positive integer, let $G$ be a non-elementary torsion-free hyperbolic group and let $H_1, \dots, H_k\leqslant G$ be quasi-convex subgroups of infinite index. There exists a malnormal quasi-convex free subgroup $H\leqslant G$ of rank $n$ such that $H_i\cap H^g = 1$ for all $i$ and all $g\in G$.
\end{theorem}

Kapovich establishes this result for free groups in order to show that hyperbolic groups contain malnormal quasi-convex free subgroups \cite{kapovich_99}. We provide an alternative proof of the free group statement and use it to prove \cref{conjugacy_separated_subgroup_intro}. 

Our motivation for establishing \cref{freiheit_intro,conjugacy_separated_subgroup_intro} is to prove an embedding result: we combine them to show that torsion-free hyperbolic and virtually compact special groups embed as quasi-convex subgroups of hyperbolic and virtually compact special groups with vanishing first $L^2$-Betti number. By combining this embedding result with the fibring theorem of Fisher \cite{fisher_21}, we prove our strongest result.

\begin{theorem}
\label{main_v2_intro}
Let $H$ be a hyperbolic and virtually compact special group with $\cd_{\Q}(H)\geqslant 2$. There is some finite index subgroup $L\leqslant H$ and a commutative diagram of exact sequences of groups
\[
\begin{tikzcd}
1 \arrow[r] & K \arrow[d, hook] \arrow[r] & L \arrow[d, hook] \arrow[r] & \mathbb{Z} \ar[equal]{d} \arrow[r] & 1 \\
1 \arrow[r] & N \arrow[r]                 & G \arrow[r]                 & \mathbb{Z} \arrow[r]                                & 1
\end{tikzcd}
\]
with the following properties:
\begin{enumerate}
\item $G$ is hyperbolic, compact special, and contains $L$ as a quasi-convex subgroup.
\item $\cd_{\Q}(G) = \cd_{\Q}(H)$.
\item $N$ is finitely generated.
\item If $b_i^{(2)}(H) = 0$ for all $2\leqslant i\leqslant n$, then $N$ is of type $\fp_n(\Q)$.
\item If $b_i^{(2)}(H) = 0$ for all $i\geqslant 2$, then $\cd_{\Q}(N) = \cd_{\Q}(H) - 1$.
\end{enumerate}
\end{theorem}

We may now prove our main theorem.

\begin{proof}[Proof of Theorem \ref{main_2}]
{Since all subgroups of free groups are free, and of cyclic groups are cyclic}, any subgroup of a free-by-cyclic group is free-by-cyclic. Hence, \ref{itm:subgroup of fbc} implies \ref{itm:vfbc}. Since $L^2$-Betti numbers are multiplicative in the index (see the book of L\"uck \cite[Theorem 6.54]{luck_02}) and finite-index subgroups and overgroups have the same rational cohomological dimension, a result Swan attributes to Serre \cite[Theorem 9.2]{Swan1969}, the fact that \ref{itm:vfbc} implies \ref{itm:zero_betti} follows from our computation of the $L^2$-Betti numbers of free-by-cyclic groups in \cref{ascending_corollary}. 

So now we prove that \ref{itm:zero_betti} implies \ref{itm:subgroup of fbc}. If $\cd_{\Q}(H) \leqslant 1$ then $H$ is virtually free by the work of Dunwoody \cite[Corollary 1.2]{dunwoody_79}. Since finitely generated free groups are by definition subgroups of \{finitely generated free\}-by-cyclic groups, the result holds in this case. So now suppose that $\cd_{\Q}(H) = 2$. By Theorem \ref{main_v2_intro} we see that $H$ has a finite index subgroup that embeds in a torsion-free group $G$ that has a finitely generated normal subgroup $N\leqslant G$ such that $G/N\isom \Z$ and $\cd_{\Q}(N) = 1$. Stallings showed that torsion-free finitely generated groups of cohomological dimension one are free \cite{sta_68}. Thus, $N$ is free and so $G$ is \{finitely generated free\}-by-cyclic. 
\end{proof}

\subsection*{Recent developments}

Since the first version of this paper appeared, there has been a number of developments that should be mentioned here. Jaikin-Zapirain~\cite{JaikinZapirain2023} proved that all torsion-free one-relator groups are homologically coherent. This was quickly followed by a proof of the second author \cite{Linton2023} that they are all in fact coherent. {Finally, Jaikin--Zapirain and the second author combined and generalised their results in \cite{JZL23}.}

In a different direction, \cref{conj} was confirmed by Fisher--S\'anchez-Peralta~\cite{FisherSanchezPeralta2023}.

{Finally, and even more recently, Fisher~\cite{Fisher2024} showed that \cref{main_2} holds with the hypotheses of being hyperbolic and virtually compact special being replaced by being virtually RFRS.}

\subsection*{Acknowledgements}
This work has received funding from the European Research Council (ERC) under the European Union's Horizon 2020 research and innovation programme (Grant agreement No. 850930).

The authors are grateful to Ismael Morales for pointing out the work of Peterson--Thom. They are also grateful to Jack Button, John Mackay, Alan Logan, Henry Wilton, and Daniel Woodhouse for their comments on the previous version of this article. {They would also like to thank the anonymous referee for the many helpful comments, especially for pointing out a simpler version of \cref{betti_extension} which is explained in \cref{simplification}.}

\section{Agrarian and \texorpdfstring{$L^2$}{L\texttwosuperior}-homology}

We will make crucial use of the theory of $L^2$-homology. We are going to present it as a special case of
agrarian homology, as introduced by Henneke and the first author in \cite{HennekeKielak2021} and then generalised by Fisher{~\cite{fisher_21}}, for two reasons.
First, for torsion-free groups satisfying the strong Atiyah conjecture, the agrarian point of view is much more straight-forward than any of the competing descriptions of $L^2$-homology. And all of the groups we will be interested in here do satisfy the strong Atiyah conjecture and are (virtually) torsion-free.
Second, the proofs of the Freiheitssatz of Peterson--Thom and of our generalisation of it work most naturally in the agrarian setup. 

We expect that the later parts of the article could also in the future be generalised to the agrarian setting. At this moment however the theory of agrarian homology is not sufficiently developed: more concretely, there is no analogue of Schreve's theorem \cite{schreve_14}*{Theorem 1.2}. 

Let us now briefly recall the necessary definitions.

\begin{definition}
	Let $R$ be a ring and $\mathcal D$ a skew-field. We will say that $G$ is $\mathcal D$-agrarian over $R$ if the group ring $RG$ embeds into $\mathcal D$ as a ring.
	
	The embedding turns $\mathcal D$ into an $RG$-bimodule, and we use this module structure to define the \emph{agrarian $\D$-homology} and \emph{$\D$-cohomology} to be
	\[
	H_i(G,\D) {= H_i(P_*\otimes_{RG}\D)} \textrm{ and } H^i(G,\D) {= H^i(\Hom_{RG}(P_*, \D))}
	\]
	respectively, {where $P_*\to R\to 0$ is any projective resolution of $R$ as a trivial $RG$-module}. We define the \emph{$i$th agrarian $\D$-Betti number} of $G$ to be $b_i^\D(G) = \dim_\D H_i(G,\D)$.
\end{definition} 

In \cite{HennekeKielak2021}, the ring $R$ was always $\Z$, however, following  Fisher, we want to allow $R$ to be a finite field. 

\begin{lemma}
	\label{hom is cohom}
	Let $H$ be a $\D$-agrarian group over a ring $R$. The following $\D$-modules are isomorphic:
	\[
	H^i(H, \D)\isom \hom_\D\big(H_i(H, \D),\D\big).
	\]
	In particular, if $b_i^\D({H})$ is finite then 
		\[
	H^i(H, \D)\isom H_i(H, \D).
	\]
\end{lemma}
\begin{proof}
	The last statement is clear, so let us focus on proving the first one.
	
	Let $P_*$ be a free resolution of $\Z$ over $\Z H$ with differential $\partial$. We have
	\[
	\hom_{\Z H}(P_i, \D)\isom \hom_{\D}(P_i\otimes_{\Z H}\D, \D).
	\]
	Since $\D$ is  {a} skew-field, every module $P_i\otimes_{\Z H}\D$ {is isomorphic to} 
	\[
	H_i(H,\D) \oplus \mathrm{im} (\partial_{i+1} \otimes_{\Z H} \mathrm{id}_\D) \oplus \mathrm{im} (\partial_i \otimes_{\Z H} \mathrm{id}_\D),
	\]
	{where  $\mathrm{im} (\partial_i \otimes_{\Z H} \mathrm{id}_\D)$ is identified with a $\D$-submodule complementary to $\ker(\partial_i\otimes_{\Z H}\mathrm{id}_{\D})$}. Applying $\hom_\D$ gives us a similar decomposition, and taking homology returns $\hom_\D\big(H_i(H, \D),\D\big)$ in dimension $i$, as claimed.
\end{proof}

Later we will need the following computation.

\begin{lemma}
	\label{L2 Betti of Fn}
	Let $F_n$ denote the free group of rank $n>0$, let $R$ be a ring and let $\D$ be a skew-field containing $RF_n$. We have
	\[
	b_i^{\D}(F_n) = \left\{ \begin{array}{cl}
	n-1 & \textrm{if } $i=1$, \\
	0 & \textrm{otherwise.}
	\end{array} \right.
	\]
\end{lemma}
\begin{proof}
	Let $\{g_1, \dots, g_n \}$ be a free generating set of $F_n$, and consider the cellular chain complex over $R$ of the Cayley graph of $F_n$ with respect to this set as the following chain complex of free $R F_n$-modules:
	\[
	{0 \to }(R F_n)^n \to R F_n{\to 0}.
	\]
	Picking the obvious bases allows us to identify the boundary map as left-multiplication with $(1-g_1, \dots, 1-g_n)$. Since $1-g_1$ is non-zero in $RF_n$, it is also non-zero in $\D$, and hence it is invertible therein. Hence after tensoring the chain complex with $\D$ over $RF_n$, we may easily change the basis of the $1$-chains over $\D$ and transform the boundary map to being left-multiplication with  $(1, 0 , \dots, 0)$. The result follows immediately.
\end{proof}

The key examples of $\D$ will be the Linnell skew-field $\D_{\Q G}$ and the Jaikin-Zapirain skew-fields $\D_{\F_p G}$. Let us discuss them in some detail.

Linnell introduced his skew-field to resolve the strong Atiyah conjecture for free groups \cite{linnell_93}. It is defined to be the division closure of $\mathbb Q G$ inside the algebra of operators affiliated to the group von Neumann algebra of $G$, and it happens to be a skew-field precisely when $G$ is torsion-free and satisfies the strong Atiyah conjecture \cite[Lemma 10.39]{luck_02}. Moreover, when it is a skew-field, the corresponding agrarian Betti numbers coincide with the $L^2$-Betti numbers. For more information we refer the reader to \cite[Section 10]{luck_02}. In general, throughout the paper we will use the combination of torsion-freeness and the strong Atiyah conjecture to guarantee the existence of $\D_{\Q G}$, and hence it will not be important what the strong Atiyah conjecture actually says.

Agol in his work on the Virtually Fibred Conjecture of Thurston introduced the class of \emph{RFRS} groups \cite[Definition 2.1]{Agol2008}. We will not need the precise definition, since we will treat the concept as a black box. Crucially for us however, RFRS groups are torsion-free and satisfy the strong Atiyah conjecture, see \cite[Proposition 4.2]{kielak_20}for details. Therefore the previous discussion applies to RFRS groups.

Let $\F$ be a skew-field and let $G$ be a RFRS group.  Jaikin-Zapi\-rain~\cite{JaikinZapirain2021} defined a skew-field $\D_{\F G}$ such that $G$ is $\D_{\F G}$-agrarian over $\F$. When $\F = \mathbb Q$, then $\D_{\mathbb Q G}$ coincides with the Linnell skew-field, and hence there is no clash of notation. Moreover, the agrarian $\D_{\F G}$-Betti numbers are multiplicative with respect to finite-index subgroups \cite[Lemma 6.3]{fisher_21}, in some cases satisfy an approximation property \`a la L\"uck \cite[Conjecture A and Theorem B]{Fisheretal2022}, and play a crucial role in controlling virtual algebraic fibring \cite[Theorem B]{fisher_21}, which makes them seem rather analogous to $L^2$-Betti numbers. In particular,  when $\F$ is the finite field $\F_p$, the reader is invited to think of  the agrarian $\D_{\F_p G}$-Betti numbers as mod $p$ variants of $L^2$-Betti numbers.

\smallskip
Let $G$ be a torsion-free group satisfying the strong Atiyah conjecture that splits as a semi-direct product $G = N \rtimes \Z$, and  let $t \in G$ be a generator of the factor $\Z$.
The subgroup $N$ also satisfies the strong Atiyah conjecture \cite[Lemma 10.4]{luck_02}. Moreover, by definition of the Linnell skew-field, we also have that $\D_{\Q N}$ is a sub-skew-field of $\D_{\Q G}$, see \cite[Section 10]{luck_02}.

 The group ring $\Q G$ is now isomorphic with the twisted Laurent polynomial ring $(\Q N)[t^{\pm 1}]$, where the twisting is induced by the conjugation action of $t$ on $N$. 
The conjugation action of $t$ extends to an action on $\D_{\Q N}$, and hence
we may form a twisted Laurent polynomial ring $\D_{\Q N}[t^{\pm 1}]$. This ring embeds into its ring of fractions (Ore localisation) $\mathrm{Ore}(\D_{\Q N}[t^{\pm 1}])$, which is nothing other than the ring of twisted rational functions in $t$ with coefficients in $\D_{\Q N}$. By \cite{luck_02}*{Lemma 10.69}, the rings $\D_{\Q G}$ and $\mathrm{Ore}\big(\D_{\Q N}[t^{\pm1}]\big)$ are isomorphic in a way commuting with the natural embedding of $\D_{\Q N}[t^{\pm1}]$.

Using the same twisting as in the construction of $\D_{\Q N}[t^{\pm1}]$, we may form the ring of twisted {formal} Laurent series $\D_{\Q N}[t^{\pm 1}\rrbracket$. {Recall that the elements of $\D_{\Q N}[t^{\pm 1}\rrbracket$ are the formal Laurent series $\sum_{i = n}^{\infty}d_it^i$ with $d_i\in \D_{\Q N}$ and $n\in \Z$}. Since $\D_{\Q N}$ is a skew-field, so is $\D_{\Q N}[t^{\pm 1}\rrbracket$.

\begin{lemma}
	\label{Laurent series}
	The natural embedding $\D_{\Q N}[t^{\pm 1}] \to \D_{\Q N}[t^{\pm 1}\rrbracket$ extends to an embedding $\D_{\Q G} \to \D_{\Q N}[t^{\pm 1}\rrbracket$.
\end{lemma}
\begin{proof}
	Since we have already seen that $\D_{\Q N}[t^{\pm1}] \injects \mathrm{Ore}(\D_{\Q N}[t^{\pm 1}])$ extends to $\D_{\Q G} \injects \mathrm{Ore}(\mathcal D_{\Q N}[t^{\pm 1}])$, we need only observe that $\mathrm{Ore}(\mathcal D_{\Q N}[t^{\pm 1}])$ embeds into  $\mathcal D_{\Q N}[t^{\pm 1}\rrbracket$ -- this is a standard property of Ore localisation.
\end{proof}

An entirely analogous discussion can be carried out for the Jaikin-Zapirain skew-fields, using \cite[Proposition 2.2(2)]{JaikinZapirain2021}.

\section{Freiheitssatz}

This section is devoted to the proof of the following agrarian variation on the $L^2$-Freiheitssatz of Peterson--Thom~\cite[Corollary 4.7]{peterson_11} (see also \cite[Theorem 5.1]{Swan1969}).

\begin{theorem}
	\label{freiheit}
Let $G$ be $\D$-agrarian over a ring $R$. Let $\{g_1, \dots, g_n\}$ be a generating set for $G$. There is a subset of the generators $\{h_1, \dots, h_k\}$ such that for every choice of a $k$-tuple of positive integers 
\[\mathbf{m} = (m_1, \dots, m_k),\]
 the subset $\{{h_1}^{m_1}, \dots, {h_k}^{m_k}\}$ freely generates a subgroup $H_\mathbf{m}$ of $G$, and the restriction  map
	\[
	H^1(G, \mathcal{D})\to H^1(H_\mathbf{m}, \mathcal{D})
	\]
	is an isomorphism.
\end{theorem}

The statement directly implies \cref{freiheit_intro}.

The proof follows that of Peterson--Thom very closely, and is based on the following observation they made. {Recall that a map $c\colon G\to \D$ is a \emph{cocycle} if $c(gh) = c(g) + gc(h)$ for all $g, h\in G$. It is \emph{inner} if there is some $d\in \D$ such that $c(g) = (g-1)d$ for all $g\in G$. Finally, recall that $H^1(G, \D)$ is the quotient of the space of cocycles by the inner cocycles.}

\begin{lemma}
	\label{PT lemma}
	Let $G$ be $\D$-agrarian over a ring $R$, and let $H \leqslant G$. If there exists a $1$-co-cycle $c \colon G \to \mathcal D$ that vanishes on $H$, and an element $g \in G$ with $c(g)\neq 0$, then the subgroup $\langle H, g \rangle$ of $G$ splits as a free product $H\ast \langle g \rangle$.
	
	Moreover, if $H$ is non-trivial then $c$ is not an inner $1$-co-cycle.
\end{lemma}
\begin{proof}
	Since the group ring $R G$ embeds into a skew-field, it cannot contain any non-trivial zero-divisors, and hence $G$ must be torsion-free. 
	As $c(g) \neq 0$, the element $g$ is non-trivial, and hence of infinite order.
	Thus $\langle g\rangle$ is isomorphic to $\Z$, and so $c$ is inner on $\langle g\rangle$, since $H^1(\Z,\D) \cong H_1(\Z,\D) = 0$ by \cref{hom is cohom,L2 Betti of Fn}.
	
	Suppose for a contradiction that $H$ and $\langle g \rangle$ do not form a free product. Then there is a minimal-length alternating word 
	\[
	h_1g^{j_1}\cdots h_mg^{j_m}
	\]
	that is trivial, but such that $h_i\neq 1$, $g^{j_i}\neq 1$ for all $i$. Conjugating the word in a suitable way shows that there is also no shorter alternating word with $h_1 = 1$ or $j_m=0$ that is trivial. 
	
As we have shown above,	 $c(g^i) = (g^i-1)d$ for some $d \in \D$ and for all $i$. Observe that $c(g) \neq 0$ and so $d \neq 0$, as $\mathcal D$ is a skew-field. We compute
	\begin{align*}
	0 &= c(h_1g^{j_1}\cdots h_mg^{j_m}) \\
	&= h_1c(g^{j_1}) + {h_1g^{j_1}h_2c(g^{j_2}) +}\dots + h_1g^{j_1}\cdots h_mc(g^{j_m}) \\
	&= \big(h_1(g^{j_1} - 1) + {h_1g^{j_1}h_2(g^{j_2}-1) +}\dots + h_1g^{j_1}\cdots h_m(g^{j_m}-1)\big)d.
	\end{align*}
	As $\mathcal{D}$ does not contain non-trivial zero-divisors, we must have
	\[
	h_1(g^{j_1} - 1) + \dots + h_1g^{j_1}\cdots h_m(g^{j_m}-1) = 0.
	\]
	Moreover, the above equation holds in the group ring $R G$, since $R G$ embeds into $\mathcal D$. 
	We conclude that at least two of the group elements appearing in the last equation must coincide; since all these words are prefixes of the word we started with, there must be a word of shorter length over $H$ and $g$ that is trivial, yielding a contradiction.
	
	Finally, assume that $H$ is non-trivial. If $c$ were inner, then we would have $c \colon h \mapsto (h-1)d$ for all $h \in H \ast \langle g \rangle$ and some $d \in \mathcal D$. Take a non-trivial $h \in H$. We have $0 = c(h) = (h-1)d$, and so $d = 0$. But then $c = 0$, a contradiction.
\end{proof}

\begin{proof}[Proof of \cref{freiheit}]
If the group $G$ is trivial then we take $k=0$ and the result is obvious. In what follows we will assume that not all of the generators $g_1, \dots, g_n$ are trivial.	
	
Consider subsets $S$ of $\{g_1, \dots, g_n\}$ with the following properties: the subset $S$ generates  a subgroup of $G$ freely, and the restriction map
		\[
	H^1(G, \D)\to H^1(\langle S \rangle, \D)
	\]
	is surjective. 
	Every subset $\{g_i\}$ with $g_i$ non-trivial satisfies the requirements, as $G$ is torsion-free and therefore $\langle g_i \rangle \cong \Z$, whose agrarian Betti numbers all vanish by \cref{L2 Betti of Fn}. Hence there is clearly a maximal subset satisfying the requirements; without loss of generality let it be $\{g_1, \dots, g_k \}$, and let $H$ denote the group it freely generates.
	
	We now claim that the restriction map
	\[
	r \colon H^1(G, \D)\to H^1(H, \D)
	\]
	is an isomorphism. If it is not the case, then there exists a {non-inner $1$-co-cycle $c \colon G \to \D$ that is inner} on $H$. {By possibly replacing $c$ with $c - c'$, where $c'$ is an appropriately chosen inner cocycle, we may assume that $c$ actually vanishes on $H$.} Since $c$ is non-zero, it does not vanish on at least one element of $\{g_1, \dots, g_n\}$, say $g_{k+1}$. The co-cycle $c$ restricts to a non-zero $1$-co-cycle $c'$ on $\langle H, g_{k+1}\rangle$. Since $R \langle H, g_{k+1}\rangle$ embeds into $\D$, as $R G$ does, we conclude using \cref{PT lemma} that $H' = \langle H, g_{k+1}\rangle$ is freely generated by $\{g_1, \dots, g_{k+1}\}$. Moreover, the $1$-co-cycle $c'$ is not inner, and hence spans a $1$-dimensional $\D$-submodule of $H^1(H', \D)$ that lies in the kernel of the restriction map
	\[
	r' \colon H^1(H', \D) \to H^1( H, \D).
	\]
	Since $r$ is onto, so is $r'$, as $r$ factors through $r'$. By \cref{L2 Betti of Fn}, we also know the agrarian Betti numbers of $H'$ and $H$, as these are finitely generated free groups. We conclude that the kernel of $r'$ is precisely the $\D$-span of $c'$. Therefore the restriction map
	\[
	H^1(G, \D)\to H^1(H', \D)
	\]
	is onto, since its image contains $c'$ and is taken by $r'$ onto $H^1(H,\D)$.
	This contradicts the maximality of the generating set of $H$, and proves our claim that $r$ is an isomorphism.
	
	\smallskip
	We now need to investigate what happens when we are presented with a tuple $\mathbf{m}$. We let $h_i = g_i$ for $i \leqslant k$, and hence $H_{(1,\dots,1)}$ coincides with our $H$. It is clear that $H_\mathbf{m}$ is always a free group of rank $k$, and hence we are left with proving that the restriction map
		\[
	r_\mathbf m \colon H^1(H, \D)\to H^1(H_\mathbf m, \D)
	\]
	is an isomorphism. Since the dimensions of both $\D$-modules are equal by \cref{L2 Betti of Fn}, as $H \cong H_\mathbf m$, it is enough to show that $r_\mathbf m$ is injective.
	
	If it is not, then there exists a non-trivial $1$-cocycle $d \colon H \to \D$ that vanishes on $H_\mathbf m$ and is not zero on some $h_i$. But then \cref{PT lemma} tells us that $h_i$ and $H_\mathbf m$ form a free product, which is not true since ${h_i}^{m_i} \in H$. This finishes the proof.
\end{proof}

\section{\texorpdfstring{$L^2$}{L\texttwosuperior}-independent hierarchies}
\label{section_4}

From now on, we focus on $L^2$-homology. We will indicate throughout what is missing in the agrarian picture.

Let $G$ be a torsion-free group satisfying the strong Atiyah conjecture, and recall that $\D_{\Q G}$ denotes the Linnell skew-field. We say that a subgroup $H\leqslant G$ is \emph{$L^2$-independent} (in $G$) if the restriction map
\[
H^1(G, \D_{\Q G})\to H^1(H, \D_{\Q G})
\]
is surjective. This definition is a modification of the one given in \cite{antolin_22}: Antolin and Jaikin-Zapirain require the co-restriction map $H_1\big(H, \mathcal U(G)\big)\to H_1\big(G, \mathcal U(G)\big)$ to be injective, where $\mathcal U(G)$ denotes the algebra of operators affiliated to the group von Neumann algebra of $G$. When $G$ is torsion-free and satisfies the strong Atiyah conjecture, the definitions coincide.

\begin{example}
\label{independent_factor}
Free factors of free groups are $L^2$-independent. More generally, retracts of free groups (and surface groups) are $L^2$-independent \cite[Proposition 4.6]{antolin_22}.
\end{example}

\begin{example}
\label{magnus_independent}
Let $G = \langle \Sigma \mid w\rangle$ be a one-relator group with $w$ cyclically reduced and not a proper power. By \cite[Corollary 1.3]{jaikin_20} of Jaikin-Zapirain--L\'opez-\'Alvarez, the strong Atiyah conjecture holds for one-relator groups without torsion. A \emph{Magnus subgroup} of $G$ is a subgroup generated by a proper subset $S\subset \Sigma$, excluding at least one generator that is mentioned in $w$. The classical Freiheitssatz of Magnus \cite{magnus_30} tells us that Magnus subgroups of one-relator groups are free. We claim that Magnus subgroups of $G$ are also $L^2$-independent.

Let $S\subset \Sigma$ be a subset, omitting a single generator $g$ that appears in $w$. Since we cannot have $G\isom \langle g\rangle*\langle S\rangle$, \cref{PT lemma} implies that there is no $1$-co-cycle $c\colon G\to \D_{\Q G}$ that vanishes on $\langle S\rangle$ and such that $c(g)\neq 0$. This implies that every $1$-co-cycle that vanishes on $\langle S\rangle$, vanishes on all of $G$. Thus, $H^1(G, \D_{\Q G})\to H^1(\langle S\rangle, \D_{\Q G})$ is injective. By \cite{dicks_07}, we have $b_1^{(2)}(G) = b_1^{(2)}(\langle S\rangle)$ and so by dimension considerations, $H^1(G, \D_{\Q G})\to H^1(\langle S\rangle, \D_{\Q G})$ is an isomorphism. In particular, $\langle S\rangle$ is $L^2$-independent. Coupled with \cref{independent_factor}, we see that all Magnus subgroups are $L^2$-independent.
\end{example}

\begin{definition}
Let $\mathcal{G}_0$ denote the class of free groups. Let $\mathcal{G}_i$ denote the class of groups $G$ satisfying the strong Atiyah conjecture, such that $G \isom A*_C$ or $G\isom A*_CB$ where $A, B\in \mathcal{G}_{i-1}$, $C$ is free, and such that $C$ is $L^2$-independent in $A$. For $n\geqslant 0$, we say that a group $G\in \mathcal{G}_n$ has an \emph{$L^2$-independent hierarchy of length $n$}. If $G\in \mathcal{G}_n$ for some $n$, then $G$ has an \emph{$L^2$-independent hierarchy}.
\end{definition}

\begin{example}
An \emph{elementary hierarchy}, as defined by Ha\-gen--Wise in \cite{hagen_10}, is a hierarchy as above in which the associated groups $C$ are always trivial or infinite cyclic. We claim that any such hierarchy is an $L^2$-independent hierarchy.

We first show that groups with an elementary hierarchy are locally indicable. It will then follow from \cite[Theorem 1.1]{jaikin_20} that they all satisfy the strong Atiyah conjecture. Firstly, note that free products of locally indicable groups are locally indicable. In particular, free groups are locally indicable. If $A$ and $B$ are locally indicable groups, then for all non-trivial elements $a\in A$, $b\in B$, the group $A*B/\normal{ab}$ is locally indicable by a result of Howie~\cite[Theorem 4.2]{howie_82}. So local indicability is closed under amalgamated free products over infinite cyclic subgroups. If $A$ is locally indicable, it follows from the same result that for all non-trivial elements $a_1, a_2\in A$, the group $A*\langle t\rangle/\normal{t^{-1}a_1ta_2}$ is locally indicable. So local indicability is closed under HNN-extensions over infinite cyclic subgroups. By induction on hierarchy length, we see that any group with an elementary hierarchy is locally indicable and so satisfies the strong Atiyah conjecture.

As $b_1^{(2)}(\Z) = b_1^{(2)}(1) = 0$ by \cref{L2 Betti of Fn} and a direct calculation, an elementary hierarchy is always an $L^2$-independent hierarchy.
\end{example}

\begin{example}
\label{limit_groups}
Limit groups that do not contain $\Z^3$ admit an elementary hierarchy, as explained in \cite{hagen_10}, and thus admit $L^2$-independent hierarchies.
\end{example}

\begin{example}
\label{ascending_example}
The main result of \cite{feighn_99} by Feighn--Handel states that a finitely generated ascending HNN-extension of a free group $G$ has a presentation of the form
\[
G\isom \langle A, B, t \mid t^{-1}at = \psi(a), a\in A\rangle 
\]
where $A$, $B$ are finite sets and $\psi\colon \langle A\rangle\to \langle A, B\rangle$ is a monomorphism of free groups. Now $G$ satisfies the strong Atiyah conjecture by Linnell's Theorem \cite[Theorem 10.19]{luck_02}. Thus, if $G$ is a finitely generated ascending HNN-extension of a free group, then $G$ has an $L^2$-independent hierarchy of length one by \cref{independent_factor}.
\end{example}

\begin{example}
Let $G$ be a torsion-free one-relator group. By \cite[Theorem 5.2]{linton_22}, there exists a sequence of subgroups 
\[G_0\leqslant G_1\leqslant \dots\leqslant G_n = G\]
 such that $G_0$ is a free group, $G_i$ is a one-relator group and $G_i = G_{i-1}*_{C_{i-1}}$ for all $i\geqslant 1$, where $C_{i-1}$ is a Magnus subgroup of $G_{i-1}$. It now follows from \cref{magnus_independent} that torsion-free one-relator groups admit $L^2$-independent hierarchies.
\end{example}

\begin{proposition}
\label{independent}
Let $G$ be a torsion-free group satisfying the strong Atiyah conjecture. Suppose that $G\isom A*_{C}$, with $A$ and $C$ finitely generated. If $C$ is $L^2$-independent in $A$, then the restriction map
\[
H^2(G, \D_{\Q G})\to H^2(A, \D_{\Q G})
\]
is injective.
\end{proposition}

\begin{proof}
We are going to use the Mayer--Vietoris sequence for HNN-extensions in group cohomology due to Bieri~\cite{Bieri1975}.  We get a long exact sequence containing
\[
\begin{tikzcd}
	& H^1(A, \D_{\Q G})\arrow[r]\arrow[d, phantom, ""{coordinate, name=Z}]& H^1(C, \D_{\Q G}) \arrow[dll,  rounded corners, to path={ -- ([xshift=2ex]\tikztostart.east)|- (Z) [near end]\tikztonodes-| ([xshift=-2ex]\tikztotarget.west)-- (\tikztotarget)}] \\
	H^2(G, \D_{\Q G})\arrow[r]& H^2(A, \D_{\Q G}).& 
\end{tikzcd}
\]
Let $N$ denote the normal closure of $A$ in $G$, and observe that $G = N \rtimes \Z$, with the factor $\Z$ being generated by $t$, the stable letter of the HNN-extension. Recall that the Linnell skew-field $\D_{\Q N}$ of $N$ is a sub-skew-field of $\D_{\Q G}$ containing $\Q N$, and hence also $\Q A$ and $\Q C$.
Using finite generation of $A$, \cref{hom is cohom}, and the fact that $\D_{\Q G}$ is a flat $\D_{\Q N}$-module (since they are both skew-fields), we see that
\begin{align*}
H^1(A, \D_{\Q G}) &\cong H_1(A, \D_{\Q G}) \\
 &\cong  H_1(A, \D_{\Q N}) \otimes_{\D_{\Q N}} \D_{\Q G}\\
 &\cong  H^1(A, \D_{\Q N}) \otimes_{\D_{\Q N}} \D_{\Q G}
\end{align*}
as $\D_{\Q G}$-modules, and similarly 
\[
H^1(C, \D_{\Q G}) \cong H^1(C, \D_{\Q N}) \otimes_{\D_{\Q N}} \D_{\Q G}.
\]
 We may therefore pick suitable bases for $H^1(A, \D_{\Q G})$ and $H^1(C, \D_{\Q G})$ in such a way that the restriction map $X$ becomes a matrix over $\D_{\Q N}$. It is now immediate that the restriction map 
 \[
 Y \colon H^1(A, \D_{\Q G}) \to H^1(C^t, \D_{\Q G})
 \]
  is also given by a matrix over $\D_{\Q N}$, where $C^t$ denotes the conjugate of $C$ by $t$. The first map of our exact sequence above is now $X - tY$.
  
The map $X$ is surjective by assumption. We now find a subset of the previously chosen basis of $H^1(A, \D_{\Q G})$, such that the restriction $X'$ of $X$ to the $\D_{\Q G}$-module spanned by this subset is a surjective square matrix. Let $Y'$ denote $Y$ restricted to the same submodule.

We are now looking at $X' -tY'$, with $X'$ and $Y'$ matrices over $\D_{\Q N}$. Since $X'$ is a surjective square matrix, it is invertible over $\D_{\Q N}$, with inverse $Z$. Now the matrix $(X' -tY')Z = \mathrm I - tY'Z$ is invertible over $\D_{\Q N}[t^{\pm 1} \rrbracket$, with inverse $\sum_{i=0}^\infty (tY'Z)^i$. Thus, the matrix $X' - tY'$ is injective over $\D_{\Q G}$, since \cref{Laurent series} tells us that $\D_{\Q G}$ is a subring of  $\D_{\Q N}[t^{\pm 1} \rrbracket$. But this means that $X' - tY'$ is an isomorphism over $\D_{\Q G}$, and hence a surjection. This finally implies that $X -tY$ is surjective over $\D_{\Q G}$ as well.

Let us return to the Mayer--Vietoris sequence.
As the first map is surjective, the second one is zero, forcing the last one to be injective, as claimed.
\end{proof}

\begin{proposition}
\label{independent_2}
Let $G$ be a torsion-free group satisfying the strong Atiyah conjecture. Suppose that $G\isom A*_CB$. If $C$ is $L^2$-independent in $A$, then the sum of the restriction maps
\[
H^2(G, \D_{\Q G})\to H^2(A, \D_{\Q G})\oplus H^2(B, \D_{\Q G})
\]
is injective.
\end{proposition}

\begin{proof}
	The proof is similar to the previous one, but easier.
Now we are going to use the Mayer--Vietoris sequence for amalgamated free products in group cohomology due to Swan~\cite[Theorem 2.3]{Swan1969}, based on a sketch of Lyndon~\cite{Lyndon1950}. Again, we look at the following exact sequence
\[
\begin{tikzcd}
&H^1(A, \D_{\Q G})\oplus H^1(B, \D_{\Q G})\arrow[r]\arrow[d, phantom, ""{coordinate, name=Z}]& H^1(C, \D_{\Q G}) \arrow[dll,  rounded corners, to path={ -- ([xshift=2ex]\tikztostart.east)|- (Z) [near end]\tikztonodes-| ([xshift=-2ex]\tikztotarget.west)-- (\tikztotarget)}] \\
H^2(G, \D_{\Q G})\arrow[r]& H^2(A, \D_{\Q G})\oplus H^2(B, \D_{\Q G}).&
\end{tikzcd}
\]
The first map  is surjective, since by assumption already the image of $H^1(A, \D_{\Q G})$ is all of $H^1(C, \D_{\Q G})$. {Thus, the second map is the zero map, and so the third map is injective.}
\end{proof}

\begin{theorem}
\label{hierarchy}
Let $G$ be a non-trivial group with an $L^2$-independent hierarchy. Then $b_i^{(2)}(G) = 0$ for all $i\neq 1$ and $\cd_{R}(G) \leqslant 2$ for any commutative ring $R$.
\end{theorem}

\begin{proof}
The proof is by induction on hierarchy length $n$. The well-known fact that free groups have cohomological dimension at most one (over any ring), combined with \cref{L2 Betti of Fn} forms the base case. Now suppose that the result is true for all groups in $\mathcal{G}_{n-1}$. 

Suppose that $G \isom A*_{C}$ where $A\in \mathcal{G}_{n-1}$ and where $C$ is $L^2$-indepen\-dent in $A$. As $H^i(A, \D_{\Q A}) = 0$ and $H^i(C, \D_{\Q A}) = 0$ for all $i> 1$, extension of scalars gives us $H^i(A, \D_{\Q G}) = 0 = H^i(C, \D_{\Q G})$, and from Bieri's long exact sequence for HNN-extensions \cite{Bieri1975} we see that $b_i^{(2)}(G) = 0$ for all $i>2$. Now \cref{independent} forces $b^{(2)}_2(G) = 0$ also. As $G$ is infinite, we have $b_0^{(2)}(G) = 0$ by \cite[Theorem 1.35]{luck_02}. If $R$ is any commutative ring, we have $\cd_{R}(C)\leqslant 1$ and $\cd_{R}(A)\leqslant 2$ by induction. Now Bieri's long exact sequence yields that $\cd_{R}(G)\leqslant 2$.

Using \cref{independent_2} and Swan's long exact sequence for amalgamated free products \cite{Swan1969}, we may similarly establish the same fact when $G\isom A*_CB$ with $A, B\in \mathcal{G}_{n-1}$ and where $C$ is $L^2$-independent in $A$. 
\end{proof}

Using \cref{ascending_example}, we obtain the following corollary.

\begin{corollary}
\label{ascending_corollary}
If $G$ is a finitely generated ascending HNN-extension of a free group, then $b^{(2)}_i(G) = 0$ for all $i\neq 1$ and $\cd_R(G)\leqslant 2$ for any commutative ring $R$.
\end{corollary}

\section{Quasi-convex subgroups of hyperbolic groups}
\label{section_5}

Recall that a hyperbolic group is \emph{non-elementary} if it is not virtually cyclic. A subgroup $H\leqslant G$ is \emph{malnormal} if $H\cap H^g = 1$ for all $g\in G - H$. The aim of this section will be to prove the following theorem which is a generalisation of \cite[Theorems C \& D]{kapovich_99} by Kapovich.

\begin{mainthm}[\ref{conjugacy_separated_subgroup_intro}]
Let $n$ be a positive integer, let $G$ be a non-elementary torsion-free hyperbolic group and let $H_1, \dots, H_k\leqslant G$ be quasi-convex subgroups of infinite index. There exists a malnormal quasi-convex free subgroup $H\leqslant G$ of rank $n$ such that $H_i\cap H^g = 1$ for all $i$ and all $g\in G$.
\end{mainthm}

In order to prove \cref{conjugacy_separated_subgroup_intro}, we first require the corresponding fact for subgroups of free groups. The following appears as \cite[Theorem D]{kapovich_99}. We provide an alternative, much shorter proof below.

\begin{proposition}
\label{malnormal_free}
Let $n$ be a positive integer, let $F$ be a finitely generated non-cyclic free group and let $F_1, \dots , F_k\leqslant F$ be a finite collection of finitely generated infinite-index subgroups. There exists a malnormal subgroup $H\leqslant F$ of rank $n$ such that ${F_i} \cap H^g = 1$ for all $i$ and all $g\in F$.
\end{proposition}

{Before embarking on the proof, we shall provide a brief idea of what we will do. We shall identify $F$ with $\pi_1(\Delta)$ for some finite graph $\Delta$ and each subgroup $F_i$ with immersions $\Gamma_i\immerses\Delta$ of compact core graphs via the $\pi_1$-functor (see work of Stallings \cite{stallings_83}). The main technical part of the proof will then involve proving the existence of an immersed path $I\immerses\Delta$ which does not lift to any $\Gamma_i\immerses\Delta$. Using this path we will then explicitly provide generators of a subgroup $H\leqslant F$ such that any loop in $\Delta$ representing an element of $H$ will contain this path as a subpath. Finally, using this fact we will verify that $H$ satisfies the claimed properties.}

\begin{proof}
	We may assume that the subgroups $F_i$ are non-trivial.
	
We first prove the result for $k=0$. Jitsukawa proved that generic tuples of $n$ elements in a non-cyclic free group generate malnormal subgroups of rank $n$ \cite[Lemmata 4 and 6]{jitsukawa_02}. Therefore, there certainly exists a malnormal subgroup of $F$ of rank $n$. One can also deduce this more explicitly using the small-cancellation conditions for graph maps studied by Wise \cite[Theorem 2.14]{wise_01}.

Now suppose that $k>0$. By the above it suffices to prove the result for $n=2$ for the following reason: if there exists a subgroup $H\leqslant F$ of rank two that is malnormal and intersects all conjugates of each $F_i$ trivially, then any malnormal subgroup of $H$ will also be malnormal in $F$ and will intersect all conjugates of each $F_i$ trivially.

Let $r\geqslant 2$ be the rank of $F$ and let $\Delta$ denote the $r$-rose, that is the graph with a single vertex and an edge for each free generator $g_i$ of $F$. In this way, combinatorial paths $I\to \Delta$, with $I$ denoting the unit interval, correspond to words over the generating set for $F$ and so we make the identification $\pi_1(\Delta) = F$. Note that all of our paths will be combinatorial, in the sense that they will send the endpoints of $I$ to vertices; they will also be non-trivial. A map of graphs is an \emph{immersion} if it is locally injective. We use the symbol $\immerses$ to denote an immersion. An immersed path $I\immerses \Delta$ corresponds to a freely reduced word over the generating set. An immersed (pointed) cycle $S^1\immerses \Delta$ corresponds to a cyclically reduced word over the generating set.

Let $\Gamma_i\immerses\Delta$ denote the connected covering space such that $\pi_1(\Gamma_i) = F_i$. As $F_i$ has infinite index, $\Gamma_i$ has infinitely many vertices and edges. Let $\Gamma = \bigsqcup \Gamma_i$, which is also a covering space. As each $F_i$ was assumed to be finitely generated (and non-trivial), there is a unique compact subgraph $\core(\Gamma)\subset \Gamma$, minimal under inclusion, which is a deformation retract of $\Gamma$, see \cite{stallings_83}. The graph $\overline{\Gamma - \core(\Gamma)}$ is a forest, and each of its connected components is a tree intersecting $\core(\Gamma)$ in a single point. Hence, an immersed path in $\Gamma$ ending outside of $\core(\Gamma)$ {cannot be extended to an immersed path ending inside of $\core(\Gamma)$}. In particular, this implies that every immersed cycle $S^1\immerses\Gamma$ must factor through $\core(\Gamma)$.

Now let $I\immerses \Delta$ be an immersed path. We say that $I\immerses\Delta$ can be \emph{completed} in $\Gamma$ if it admits a lift and an extension to an immersion $S^1\immerses\Gamma$. By the above observations, an immersed path $I\immerses \Delta$ can be completed in $\Gamma$ only if it admits a lift to $\core(\Gamma)$. Let $v_1, \dots, v_m$ be some enumeration of the vertices of $\core(\Gamma)$. We will now inductively construct immersed paths $I_i\immerses\Delta$ such that the lifts of $I_i$ starting at $v_1, \dots, v_i$ end outside of $\core(\Gamma)$. Then $I_m\immerses\Delta$ will be a path that cannot be completed. Since each component of $\Gamma$ contains a vertex not in $\core(\Gamma)$, there is some path $I_1\immerses\Delta$, such that its unique lift to $\Gamma$ starting at $v_1$ ends at a vertex not in $\core(\Gamma)$. Now suppose we have already constructed $I_i\immerses \Delta$. If the lift of $I_i$ starting at $v_{i+1}$ is not contained in $\core(\Gamma)$, then we let $I_{i+1} = I_i$. Otherwise, let $\Gamma_j$ be the component containing $v_{i+1}$ and let $e$ be the last oriented edge that the lift of $I_i$ starting at $v_{i+1}$ traverses. Let $v\in \Gamma_j$ be a vertex not in $\core(\Gamma)$. If the endpoint of $e$ and $v$ lie in the same connected component of $\Gamma_j - e$, then there is an immersed path $I'\immerses \Gamma_j - e$ connecting the endpoint of $e$ with $v$. Then concatenating $I_i$ with the projection of $I'$ to $\Delta$ gives us the required path. Now suppose that $\Gamma_j - e = \Gamma_j'\sqcup\Gamma_j''$ consists of two connected components with $v\in \Gamma_j'$ and the endpoint of $e$ in $\Gamma_j''$. Since $e$ has its endpoint in $\core(\Gamma_j)$, we have that $\Gamma_j''\cap \core(\Gamma_j)$ is not contractible. Hence, there is an immersed loop $I''\immerses\Gamma_j''$ based at the endpoint of $e$. As $\overline{\Gamma_j'\cup e}$ is connected, there is an immersed path $I'\immerses \overline{\Gamma_j'\cup e}$ connecting the endpoint of $e$ with $v$. Now concatenating $I_i$ with the projection of the concatenation $I''*I'$ to $\Delta$ gives us the required path.

Now let $f\in F = \pi_1(\Delta)$ be the element corresponding to $I_m$. By extending $I_m$ if necessary, we may assume that $f$ begins with $g_1$ and ends with $g_2$. Therefore, the path associated with $f^i$ is also immersed for all $i\neq 0$ and also cannot be completed in $\Gamma$. Let $p\geqslant 3$ be greater than the maximal positive integer $i$ such that $g_1^{\pm i}$ appears as a subword of $f$. Let $q\geqslant 3$ be greater than the maximal positive integer $i$ such that $g_2^{\pm i}$ appears as a subword of $f$. Put 
\[
H = \langle g_2g_1g_2^qf^2g_1g_2, \, g_1g_2f^2g_1^pg_2g_1\rangle.
\]
We claim that $H$ is the required subgroup. Note that any cyclically reduced word over the chosen generators for $H$ is also a cyclically reduced word over the generators for $F$. Denote by $\Lambda\immerses \Delta$ the connected  cover with $\pi_1(\Lambda) = H$. Then by construction, if $S^1\immerses\Lambda$ is any immersed cycle, there must be an immersed path with label $f^2$ factoring through it. Hence, if $I\immerses \core(\Lambda)$ is any immersed loop, there must be an immersed path with label $f$ or $f^{-1}$ factoring through it. Thus, the projection of any immersed loop $I\immerses\Lambda$ to $\Delta$ cannot lift to a loop in $\Gamma$. But this precisely means that $H\cap {F_i}^g = 1$ for all $g\in F$ and all $i$. 

It remains to be shown that $H$ is malnormal. By definition of $H$, there is only one lift of the path labeled $g_1^{\pm p}$ to $\core(\Lambda)$. Equivalently, there is only one lift of the path labeled $g_1^{\pm p}$ to $\Lambda$ that can be extended to an immersed cycle. Similarly for $g_2^{\pm q}$. Moreover, if $I\immerses\Lambda$ is any immersed loop starting at the basepoint, there must be an immersed path in $\core(\Lambda)$ with label $g_1^{\pm p}$ or $g_2^{\pm q}$ factoring through it. Hence, any path $I\immerses \Delta$ that lifts to a loop in $\Lambda$, starting at the basepoint, cannot lift to any other loop in $\Lambda$. This implies that $H$ is malnormal and so we are done.
\end{proof}

The following proposition is a direct consequence of \cite[Proposition 6.7]{kharlampovich_17} by Kharlampovich--Miasnikov--Weil. A slight modification of the proof of \cite[Lemma 1.2]{gitik_98} by Gitik--Mitra--Rips--Sageev also yields the same result.

\begin{proposition}
\label{quasi-convex_intersection}
Let $G$ be a torsion-free hyperbolic group and let $K, H\leqslant G$ be two quasi-convex subgroups. There is a finite family of quasi-convex subgroups $K_1, \dots, K_m\leqslant K$, each equal to $K\cap H^g$ for some $g\in G$, such that for all $h\in G$ the group $K\cap H^h$ is conjugate within $K$ to some $K_i$.
\end{proposition}

If $G$ is a group and $H\leqslant G$ is a subgroup, then the \emph{commensurator} of $H$ in $G$ is defined to be the following subgroup:
\[
\comm_G(H) = \{g\in G \mid H\cap H^g \text{ has finite index in $H$ and $H^g$}\}. 
\]
We shall need the following result due to Arzhantseva \cite{arzhantseva_01}*{Theorem 2}.

\begin{theorem}
\label{qc_commensurator}
Let $G$ be a hyperbolic group and $H\leqslant G$ an infinite quasi-convex subgroup. Then $[\comm_G(H):H]<\infty$. In particular, $\comm_G(H)$ is quasi-convex.
\end{theorem}

We require an additional fact about commensurators of quasi-convex subgroups of hyperbolic groups.

\begin{lemma}
\label{commensurator}
Let $G$ be a hyperbolic group and $H\leqslant G$ an infinite quasi-convex subgroup. If $H\cap H^g$ has finite index in $H$, then it has finite index in $H^g$ also. In particular, we have
\[
\comm_G(H) = \{g\in G \mid H\cap H^g \text{ has finite index in $H$ or $H^g$}\}.
\]
\end{lemma}

\begin{proof}
Suppose that there is some $g\in G$ such that $H\cap H^g$ has finite index in $H$. By induction, one shows that 
\[
H\cap H^g\cap \dots \cap H^{g^n} = H \cap (H\cap H^g\cap \dots \cap H^{g^{n-1}})^g
\]
 has finite index in $H$ for all $n\geqslant 1$. By the main theorem in \cite{gitik_98}, this implies that $g^n\in H$ for some $n\geqslant 1$. But this then implies that $H\cap H^{g^{-1}} = H \cap H^{g^{n-1}}$ has finite index in $H$ and so $H\cap H^g$ has finite index in $H^g$.
\end{proof}

A rank-one version of \cref{conjugacy_separated_subgroup_intro}, stated below, is a direct consequence of \cite[Theorem 1]{minasyan_05} of Minasyan.

\begin{lemma}
\label{cyclic_conjugates}
Let $G$ be a hyperbolic group and let $H_1, \dots, H_k\leqslant G$ be quasi-convex subgroups of infinite index. There exists an element $x\in G$ of infinite order such that $\langle x\rangle \cap {H_i}^g = 1$ for each $i$ and all $g\in G$.
\end{lemma}

Finally, the following fact will be useful. It is a consequence of \cite[Theorem 5]{minasyan_06} and can also be derived from \cite{arzhantseva_01}.

\begin{proposition}
\label{qc_powers}
Let $G$ be a torsion-free hyperbolic group and let $H\leqslant G$ be a quasi-convex subgroup. Let $g\in G$ be a non-trivial element such that $\langle g\rangle\cap H =1$. Then, for sufficiently large $m\geqslant 1$, we have that $\langle H, g^m\rangle\isom H*\langle g^m\rangle$ is quasi-convex.
\end{proposition}

We are now ready to prove the main result of this section.

\begin{proof}[Proof of \cref{conjugacy_separated_subgroup_intro}]
If $G$ is free, then the result is \cref{malnormal_free}, so let us assume that $G$ is not free. As $G$ is torsion-free, it is also not virtually free by \cite[Theorem B]{Swan1969}. 

Let $x\in G$ be an element of infinite order such that $\langle x\rangle\cap {H_i}^g = 1$ for each $i$ and all $g\in G$, given by \cref{cyclic_conjugates}. Note that $\langle x\rangle$ is quasi-convex, since all cyclic subgroups of hyperbolic groups are quasi-convex. 
If for every $y \in G$ we have $\langle y \rangle \cap \langle x \rangle \neq 1$, then {$G$ does not contain any non-abelian free subgroups and thus is virtually cyclic (for example, see \cite[Remark 3.1.A(b)]{gromov_87})}, a contradiction.
Hence,
 by \cref{qc_powers} there is some element $y\in G$ such that $\langle x, y\rangle$ is isomorphic to the free group of rank two and is quasi-convex in $G$. Let $F = \comm_G(\langle x, y\rangle)$. As $G$ is not virtually free, $\langle x, y\rangle$ has infinite index in $G$. By \cref{qc_commensurator}, $F$ is quasi-convex, it is its own commensurator, and also has infinite index in $G$. Moreover, $F$ is free as it contains $\langle x, y\rangle$ as a subgroup of finite index.

By applying \cref{quasi-convex_intersection} twice -- once with $F$ in the place of $K$, and once with $F$ in the place of $K$ and of $H$ -- and taking the union of the two families of subgroups, we find that there is a finite collection of finitely generated subgroups $F_1, \dots, F_m\leqslant F$ such that:
\begin{enumerate}
\item \label{item 1}If $F\cap {H_i}^g \neq 1$ for some $i$ and some $g\in G$, then $F\cap {H_i}^g$ is conjugate in $F$ to some $F_j$,
\item \label{item 2} If $F\cap F^g\neq 1$ for some $g\notin F$, then $F\cap F^g$ is conjugate in $F$ to some $F_j$.
\end{enumerate}
By \cref{commensurator}, $F\cap F^g$ has infinite index in $F$ for all $g\notin F$. As $x\in F$, we also have that $F\cap {H_i}^g$ has infinite index in $F$ for each $i$ and all $g\in G$. Now \cref{malnormal_free} produces for us a subgroup $H$ of $F$ that is free of rank $n$, malnormal in $F$, and satisfies $H \cap {F_i}^g = 1$ for all $g \in F$. By \eqref{item 2}, the subgroup $H$ is in fact malnormal in $G$, and  by \eqref{item 1}, it satisfies $H \cap {H_i}^g = 1$ for all $g \in G$. Finally, $H$ is quasi-convex in $G$ since $F$ is. We conclude that $H$ is the desired free subgroup.
\end{proof}

The following corollary is central to the proof of our main results.

\begin{corollary}
\label{quasi-convex_subgroup}
Let $G$ be a non-elementary torsion-free hyperbolic group that satisfies the strong Atiyah conjecture. There exists a pair of isomorphic quasi-convex free subgroups $H, F\leqslant G$, such that $F$ is malnormal, ${H\cap F^g} = 1$ for all $g\in G$, and the restriction map
\[
H^1(G, \D_{\Q G})\to H^1(H, \D_{\Q G})
\]
is an isomorphism.
\end{corollary}

\begin{proof}
By \cref{freiheit}, there exists a finite set of elements $h_1, \dots, h_l\in G$ that generate a free subgroup of rank $l$, such that for every $l$-tuple $\mathbf m = (m_1, \dots,m_l)$ of positive integers, the map 
\[
H^1(G, \D_{\Q G})\to H^1(H_{\mathbf m}, \D_{\Q G})
\]
is an isomorphism, where $H_{\mathbf m} = \langle h_1^{m_1}, \dots, h_l^{m_l}\rangle$. We pick $m_1 = 2$. By \cref{qc_powers} and induction, $H_{\mathbf m}$ is quasi-convex in $G$ for some choice of $\mathbf m$; the induction starts with the statement that cyclic subgroups of $G$ are quasi-convex. Since $m_1 = 2$, the subgroup $H = H_\mathbf m$ is of infinite index in $G$. The result now follows from \cref{conjugacy_separated_subgroup_intro}, by taking $k=1, H_1 = H =  H_\mathbf m$, and $n=l$.
\end{proof}

\section{Characterising virtually free-by-cyclic groups}

We are now ready to prove the strongest result of the article, which is the cornerstone of
 \cref{main_2}.

\begin{mainthm}[\ref{main_v2_intro}]
Let $H$ be a hyperbolic and virtually compact special group with $\cd_{\Q}(H)\geqslant 2$. There is some finite index subgroup $L\leqslant H$ and a commutative diagram of exact sequences of groups
\[
\begin{tikzcd}
1 \arrow[r] & K \arrow[d, hook] \arrow[r] & L \arrow[d, hook] \arrow[r] & \mathbb{Z} \ar[equal]{d} \arrow[r] & 1 \\
1 \arrow[r] & N \arrow[r]                 & G \arrow[r]                 & \mathbb{Z} \arrow[r]                                & 1
\end{tikzcd}
\]
with the following properties:
\begin{enumerate}
\item $G$ is hyperbolic, compact special, and contains $L$ as a quasi-convex subgroup.
\item $\cd_{\Q}(G) = \cd_{\Q}(H)$.
\item $N$ is finitely generated.
\item If $b_i^{(2)}(H) = 0$ for all $2\leqslant i\leqslant n$, then $N$ is of type $\fp_n(\Q)$.
\item If $b_i^{(2)}(H) = 0$ for all $i\geqslant 2$, then $\cd_{\Q}(N) = \cd_{\Q}(H) - 1$.
\end{enumerate}
\end{mainthm}

Before proceeding with the proof, we first discuss some of the ingredients going into it. The key result we will be using is the following, due to Fisher \cite{fisher_21}, building on \cite{kielak_20}.

\begin{theorem}
\label{fisher_fibring}
Let $G$ be a virtually RFRS group of type $\fp_n(\Q)$. Then the following are equivalent:
\begin{enumerate}
\item $b_i^{(2)}(G) = 0$ for all $i\leqslant n$.
\item There is a subgroup $H\leqslant G$ of finite index that contains a normal subgroup $K\leqslant H$ such that $H/K\isom \Z$ and $K$ is of type $\fp_n(\Q)$.
\end{enumerate}
\end{theorem}

As discussed previously, we shall be using the RFRS property as a black box and therefore do not define it here. The relevant class of examples of (virtually) RFRS groups the reader should have in mind are \emph{compact special groups}. Introduced in \cite{haglund_08} by Haglund and Wise, compact special groups are fundamental groups of compact non-positively curved cube complexes whose hyperplanes behave in a controlled manner. By \cite[Theorem 1.1]{haglund_08}, compact special groups are subgroups of right-angled Artin groups, which in turn are virtually RFRS by a result of Agol \cite[Corollary 2.3]{Agol2008} (they are in fact RFRS, but we will not need this). Therefore, \cref{fisher_fibring} applies to any  virtually compact special group. 

Another important fact about virtually compact special groups we shall need is that they satisfy the strong Atiyah conjecture by \cite{schreve_14}. We collect below some of the statements we have mentioned. Note that it is precisely the lack of a suitable replacement for Schreve's theorem that stops us from formulating the following results for agrarian homology with coefficients being Jaikin-Zapirain's skew-fields. Fisher's theorem above works in this general context.

\begin{proposition}
\label{special_facts}
Let $G$ be a virtually compact special group. Then $G$ satisfies the strong Atiyah conjecture and is virtually RFRS. In particular, it has a torsion-free finite-index subgroup that is compact special and RFRS.
\end{proposition}

In order to prove \cref{main_v2_intro}, we embed a finite index subgroup of $H$ into a virtually compact special group $G$ with vanishing $L^2$-Betti numbers. However, in order to do this, virtual compact specialness of $H$ does not suffice and so we must add a hyperbolicity assumption. This is needed in order to apply the following combination theorem, obtained from Wise \cite[Theorem 13.1]{wise_21_quasiconvex} and Kharlampovich--Myasnikov \cite[Theorem 5]{kharlampovich_98}.

\begin{theorem}
\label{combination_theorem}
Let $H$ be hyperbolic and virtually compact special, let $A, B\leqslant H$ be quasi-convex subgroups and let $\psi\colon A\to B$ be an isomorphism. If $A$ is malnormal and $A\cap B^g = 1$ for all $g\in H$, then 
\[
H*_A \isom \langle H, t \mid t^{-1}at = \psi(a), a\in A\rangle.
\]
is hyperbolic, virtually compact special and contains $H$ as a quasi-convex subgroup.
\end{theorem}

We now prove our embedding result.

\begin{proposition}
\label{betti_extension}
Let $H$ be a non-free torsion-free hyperbolic compact special group with $b_1^{(2)}(H)\neq 0$. There exists a hyperbolic virtually compact special group $G$, isomorphic to an HNN-extension of $H$ with finitely generated quasi-convex free associated subgroups, such that the following holds:
\[
b_i^{(2)}(G) =
\begin{cases}
0 & \text{ if $i = 1$}\\
b_i^{(2)}(H) & \text{ if $i\neq 1$.}
\end{cases}
\]
Moreover, $H$ is a quasi-convex subgroup of $G$ and, for any commutative ring $R$, we have $\cd_R(G) = \cd_R(H)$.
\end{proposition}

\begin{proof}
By \cref{special_facts}, $H$ satisfies the strong Atiyah conjecture. As $H$ is torsion-free and non-free, it is non-elementary. By \cref{quasi-convex_subgroup}, there is some quasi-convex non-trivial free subgroup $B\leqslant H$, necessarily of infinite index, such that the restriction map 
\[
H^1(H, \mathcal{D}_{\Q H})\to H^1(B, \mathcal{D}_{\Q H})
\]
 is an isomorphism. In particular, $B$ is $L^2$-independent in $H$. Also by \cref{quasi-convex_subgroup}, there is a subgroup $A$, isomorphic to $B$, such that $A$ is malnormal and no conjugate of $A$ intersects $B$ non-trivially. Let $\psi\colon A\to B$ be any isomorphism and let 
\[
G = H*_A \isom \langle H, t \mid t^{-1}at = \psi(a), a\in A\rangle.
\]
By \cref{combination_theorem}, $G$ is hyperbolic and virtually compact special. By \cref{special_facts}, $G$ satisfies the strong Atiyah conjecture.

From \cite[Theorem 3.1]{Bieri1975}, we obtain a long exact sequence
\[
\begin{tikzcd}
0 \arrow[r] & H^0(G, \D_{\Q G}) \arrow[r]\arrow[d, phantom, ""{coordinate, name=Z}] & H^0(H, \D_{\Q G}) \arrow[r] &H^0(A, \D_{\Q G}) \arrow[dll,  rounded corners, to path={ -- ([xshift=2ex]\tikztostart.east)|- (Z) [near end]\tikztonodes-| ([xshift=-2ex]\tikztotarget.west)-- (\tikztotarget)}] \\
 &H^1(G, \D_{\Q G}) \arrow[r]\arrow[d, phantom, ""{coordinate, name=Z}] & H^1(H, \D_{\Q G}) \arrow[r] & H^1(A, \D_{\Q G}) \arrow[dll,  rounded corners, to path={ -- ([xshift=2ex]\tikztostart.east)|- (Z) [near end]\tikztonodes-| ([xshift=-2ex]\tikztotarget.west)-- (\tikztotarget)}] &\\
& H^2(G, \D_{\Q G}) \arrow[r] & H^2(H, \D_{\Q G}) \arrow[r] & H^2(A, \D_{\Q G}) \arrow[r] & \dots
\end{tikzcd}
\]
Since $A$ is free, we have $H^i(A, \D_{\Q G}) = 0$ for all $i\geqslant 2$, inducing isomorphisms between $H^i(G, \D_{\Q G})$ and $H^i(H, \D_{\Q G})$ for all $i\geqslant 3$. Moreover, all of the cohomology groups on the top row are trivial as $G$, $H$ and $A$ are infinite \cite[Theorem 6.54]{luck_02}. Hence, we are left with the exact sequence
\[
\begin{tikzcd}
0\arrow[r] &H^1(G, \D_{\Q G})\arrow[r]\arrow[d, phantom, ""{coordinate, name=Z}]  &H^1(H, \D_{\Q G})\arrow[r] &H^1(A, \D_{\Q G})\arrow[dll,  rounded corners, to path={ -- ([xshift=2ex]\tikztostart.east)|- (Z) [near end]\tikztonodes-| ([xshift=-2ex]\tikztotarget.west)-- (\tikztotarget)}] \\
 &H^2(G, \D_{\Q G})\arrow[r] &H^2(H, \D_{\Q G})\arrow[r] &0
\end{tikzcd}
\]
and equalities
\[
b_i^{(2)}(G) = b_i^{(2)}(H)
\]
for all $i\geqslant 3$. The isomorphism
\[
H^1(H, \mathcal{D}_{\Q H}) \cong H^1(B, \mathcal{D}_{\Q H})
\] 
implies that $b_1^{(2)}(H) = b_1^{(2)}(B)$. But the latter is equal to $b_1^{(2)}(A)$, since $A$ and $B$ are isomorphic. Hence an Euler-characteristic-type count reveals that 
\[
b_1^{(2)}(G) + b_2^{(2)}(H) = b_2^{(2)}(G).
\]
\cref{independent} implies that $b_2^{(2)}(G) \leqslant b_2^{(2)}(H)$, and so $b_1^{(2)}(G) = 0$ and
$b_i^{(2)}(G) = b_i^{(2)}(H)$ for all $i\neq 1$, as desired.

\smallskip
Now let us consider Bieri's long exact sequence once more, this time in cohomology over an arbitrary $RG$ module $M$. As $A$ is free, we have that $H^i(A, M) = 0$ for all $i\geqslant 2$. This implies that the maps $H^i(G, M)\to H^i(H, M)$ are isomorphisms for all $i\geqslant 3$. As the only torsion-free finitely generated groups of cohomological dimension one are free groups \cite{sta_68}, we have that $\cd_R(H)\geqslant 2$. Since $G$ contains $H$, the group $G$ cannot be free either, and so $\cd_R(G)\geqslant 2$. Finally, as $M$ was arbitrary, this implies that $\cd_R(G) = \cd_R(H)$.
\end{proof}

\begin{remark}
\label{simplification}
{There is also a version of \cref{betti_extension} for amalgamated free products which was suggested to us by the anonymous referee. If $K$ is any non-elementary hyperbolic and virtually compact special group with $b_i^{(2)}(K) = 0$ for all $i$ (for example, any atoroidal \{finitely generated free\}-by-$\Z$ group will do), then, just as in the proof of \cref{betti_extension}, we can produce a malnormal and quasi-convex free subgroup $F\leqslant H, K$ such that $G = H*_FK$ satisfies all the same conclusions of \cref{betti_extension}.}
\end{remark}

\begin{proof}[Proof of \cref{main_v2_intro}]
Firstly, we note that $H$ is non-elementary since $\cd_{\Q}(H)\geqslant 2$. Hence $b_0^{(0)}(H) = 0$ by \cite[Theorem 6.54]{luck_02}. As hyperbolicity passes to finite-index subgroups, by \cref{special_facts} there is a finite index subgroup $H'\leqslant H$ that is torsion-free, hyperbolic, and compact special. By \cite[Theorem 6.54]{luck_02}, $L^2$-Betti numbers of finite-index subgroups are multiplicative in the index. This implies that for any given $i$, we have $b_i^{(2)}(H') = 0$ if and only if $b_i^{(2)}(H) = 0$. By \cite[Theorem 9.1]{Swan1969}, we have $\cd_{\Q}(H') = \cd_{\Q}(H)$. 

\smallskip
Suppose first that $b_{1}^{(2)}(H) = 0$. Since $H'$ is torsion-free and hyperbolic, it is of type $\fp(\Q)$. By \cref{special_facts}, we may now apply \cref{fisher_fibring} to $H'$ to conclude that there is some finite-index hyperbolic and compact special subgroup $L\leqslant H'$ and a normal subgroup $K\leqslant L$ such that $L/K\isom \Z$ and such that $K$ is finitely generated and of type $\fp_n(\Q)$ precisely when $b_i^{(2)}(H) = 0$ for all $i\leqslant n$. If $b_i^{(2)}(H) = 0$ for all $i$, then as $\cd_{\Q}(K)<\infty$, we have that $K$ is of type $\fp(\Q)$. By Feldman's theorem~\cite[Theorem 2.4]{feldman_71}, we have $\cd_{\Q}(K) = \cd_{\Q}(H) - 1$. Now the result follows by setting $G = L$ and $N = K$. 

\smallskip
Now assume that $b_1^{(2)}(H) \neq 0$. By Proposition \ref{betti_extension}, there is some hyperbolic and virtually compact special group $G'$, isomorphic to an HNN-extension of $H'$ with finitely generated free associated subgroups, such that $b_i^{(2)}(G') = b_i^{(2)}(H')$ for all $i\neq 1$, $b_1^{(2)}(G) = 0$ and $\cd_{\Q}(G') = \cd_{\Q}(H)$. Moreover, $H'$ is quasi-convex in $G'$. As above, we may produce a finite index subgroup $G\leqslant G'$ and a normal subgroup $N\leqslant G$ with the properties required by the statement. Let $\phi \colon G \to G/N  \cong \Z$ denote the quotient map. All that remains to show is that the desired diagram exists. We may need to pass to a further finite-index subgroup of $G$ that contains $N$; this will not alter any of the properties we require.

We are now going to use some elementary tools from Bass--Serre theory. The reader is directed to Serre's monograph \cite{serre_80} for the necessary background material. Since $G'$ splits as an HNN-extension over $H'$, it acts on the Bass--Serre tree $T$ associated with this splitting, and $T/G'$ is the corresponding graph of groups with a single edge. Now $G$ also acts on $T$ and the quotient $T/G$ yields a graph of groups decomposition of $G$. Since $G$ has finite index in $G'$, $T/G$ is a finite graph. The stabilisers in $G'$ of vertices in $T$ are conjugates of $H'$ and so the vertex groups associated with the induced graph of groups decomposition of $G$ must all be intersections of conjugates of $H'$ with $G$. Denote by $J$ the normal closure of these vertex groups in $G$. We have $G/J\isom \pi_1(T/G)$ which is a free group of finite rank. 

Suppose that the intersection of every conjugate of $H'$ with $G$ is a subgroup of $N$. Then we have $J\leqslant N$. Hence, $\phi$ induces an epimorphism  $\pi_1(T/G)\to \Z$, and $N$ maps onto the kernel of this epimorphism. If this kernel is non trivial, then it cannot be finitely generated as it is an infinite-index normal subgroup of a free group. This would then imply that $N$ surjects a non finitely generated free group which is not possible as $N$ is finitely generated. It follows that $\pi_1(T/G)\isom \Z$ and hence that $J = N$. But this implies that $N\leqslant \normal{H'}\cap G$. Since $G/N\isom \Z$, we must also have $N = \normal{H'}\cap G$. Since $G'$ splits as a HNN-extension over $H'$ with proper associated subgroups, we have that $H'\neq \normal{H'}$. It is now a straightforward consequence of Bass--Serre theory that $\normal{H'}$ cannot be finitely generated. As $\normal{H'}\cap G$ is a finite index subgroup of $\normal{H'}$, this implies that $N$ is also not finitely generated which is a contradiction. Hence there must be some conjugate of $H'$ that does not intersect $G$ in $N$. Let $L$ be the intersection of such a conjugate with $G$. As $L$ has finite index in some conjugate of $H'$, it is also hyperbolic and compact special and so is of type $\fp(\Q)$. Finally, $L$ is quasi-convex in $G$ since all conjugates of $H'$ are quasi-convex in $G'$. After possibly passing to a finite index subgroup of $G$, we may assume that $L\injects G\to G/N = \Z$ is surjective. Then, letting $K\leqslant L$ be the kernel of this map we obtain the desired diagram.
\end{proof}

When $n=2$, the use of Fisher's theorem can be replaced by \cite[Theorem 5.4]{kielak_20}

\begin{corollary}
Let $H$ be a hyperbolic and virtually compact special group satisfying $\cd_{\Q}(H) = n\geqslant 2$ and $b_i^{(2)}(H) = 0$ for all $i\geqslant 2$. Then there is some finite-index subgroup $H'\leqslant H$ such that $H'$ splits as an HNN-extension with finitely generated associated subgroups and with finitely generated base group $A$ satisfying $\cd_{\Q}(A) = n-1$.
\end{corollary}
\begin{proof}
\cref{main_v2_intro} gives us a finite index subgroup $H'$ of $H$ that embeds into a group $G = N \rtimes \Z$ with $\cd_{\Q}(N) = n-1$. Let $\phi \colon G \to \Z$ denote the projection. Observe that $H'$ is hyperbolic, and hence of type $\fp(\Q)$. Applying  \cite[Theorem A]{bieri_78} to the restriction $\phi\vert_{H'}$, we may write $H'$ as an HNN-extension with finitely generated associated groups, and with base group $A$ finitely generated and being a subgroup of $N$. This latter property implies that $\cd_{\Q}(A) \leqslant  n-1$.
Finally, Bieri's Mayer--Vietoris sequence for HNN-extensions, applied as we did in the proof of \cref{main_v2_intro}, shows that $\cd_{\Q}(A) =  n-1$.
\end{proof}

\begin{theorem}
\label{hierarchy_2}
If $G$ is a hyperbolic and virtually compact special group with an $L^2$-independent hierarchy, then $G$ is virtually free-by-cyclic.
\end{theorem}

\begin{proof}
This is \cref{hierarchy} combined with \cref{main_2}.
\end{proof}

\begin{proof}[Proof of Corollary \ref{ascending}]
This follows from combining \cref{main_2} with \cref{ascending_corollary}.
\end{proof}

We conclude the article with a conjecture.

\begin{conjecture}
	\label{conj}
	Let $\F$ be a skew-field.
	 \cref{main_v2_intro} remains valid if every occurrence of $\cd_{\Q}$ is replaced by $\cd_{\F}$, and  every $L^2$-Betti number $b_i^{(2)}$ is replaced by the agrarian Betti number $b_i^{\D_{\F H}}$.
\end{conjecture}
	
	There is a slight caveat here, since in the formulation of \cref{main_v2_intro}, the group $H$ does not have to be RFRS, or even locally indicable, and hence $\D_{\F H}$ is not defined. It is however defined for a finite-index subgroup, and a result of Fisher~\cite{fisher_21}*{Lemma 6.3} shows that one can define agrarian Betti numbers over Jaikin-Zapirain's skew-fields by passing to finite-index subgroups, and then rescaling by the index.

\bibliographystyle{amsalpha}
\bibliography{bibliography}

\end{document}